\documentclass{article}
\usepackage{amsfonts}
\usepackage{amsthm}
\usepackage[pdftex]{graphicx}
\pdfoutput=1
\usepackage{latexsym}
\usepackage{amssymb}
\usepackage{amstext}
\usepackage{amsmath}

\usepackage{graphicx}
\usepackage{amssymb}
\usepackage{epstopdf}
\theoremstyle{plain}
\newtheorem{thm}{Theorem}[section]

\theoremstyle{definition}
\newtheorem{defn}{Definition}[section]

\def\G{\overline{G}}
\def\H{\overline{H}}
\def\C{\widehat{C}_5}
\def\Gx{\overline{G}_\times}

\def\e{\epsilon}

\begin{document}

\begin{center}
\LARGE {\bf \textsc{Homomorphic Preimages of Geometric Cycles} }
\end{center}\bigskip

\begin{center}
\textsc{Sally Cockburn}

\textsc{Department of Mathematics}

\textsc{Hamilton College, Clinton, NY 13323}

\textsc {\sl scockbur@hamilton.edu}
\end{center}


\begin{abstract}
A graph $G$ is a  homomorphic preimage of another graph $H$, or equivalently $G$ is $H$-{\it colorable},  if there exists a graph homomorphism $f:G \to H$. A classic problem is to characterize the family of homomorphic preimages of a given graph $H$. 
A \emph{geometric graph} $\overline{G}$ is a simple graph $G$ together with a straight line drawing of $G$ in the plane with the vertices in general position
A geometric homomorphism (resp. isomorphism) $\overline{G} \to \overline{H}$ is a graph homomorphism (resp. isomorphism) that preserves edge crossings (resp. and non-crossings).
The homomorphism poset $\mathcal{G}$ of a graph $G$ is the set of isomorphism classes of geometric realizations of  $G$  partially ordered by the existence of injective geometric homomorphisms.
A geometric graph $\overline{G}$ is $\mathcal{H}$-colorable if $\G\to \H$ for some $\H \in \mathcal{H}$.
In this paper, we provide necessary and sufficient conditions for $\G$ to be $\mathcal{C}_n$-colorable for $2 \leq n \leq 5$.
\end{abstract}

\section{Basic Definitions}

A graph homomorphism $f:G \to H$ is a vertex function such that  for all $u, v \in V(G)$,
	 $uv \in E(G)$ implies $ f(u)f(v) \in E( H)$.
If such a function exists,  we write $G \to H$ and say that $G$ is homomorphic to $H$, or equivalently, that $G$ is a homomorphic preimage of $H$.	
A proper $n$-coloring of a graph $G$ is  a homomorphism $G \to K_n$; thus, $G$ is $n$-colorable if and only if  $G$ is a homomorphic preimage of $K_n$.
(For an excellent overview of the theory of graph homomorphisms, see~\cite{HN}.)
\medskip

In 1981, Maurer, Salomaa and Wood  \cite{MSW1} generalized this notion  by defining  $G$ to be $H$-colorable if and only if  $G \to H$. 
They used the notation $\mathcal{L}(H)$ to denote the family of $H$-colorable graphs.  
For example, $G$ is $C_5$-colorable if and only if $G \to C_5$;  this means  there exists a proper 5-coloring of $G$ such a vertex of color 1 can only be adjacent to vertices of color 2 or 5, but not to vertices of color 3 or 4, {\it  etc}.  
Maurer {\it et al.} noted that for odd $m$ and $n$, $C_m$ is $C_n$-colorable ({\it i.e.} $C_m \to C_n$) if and only if $m \geq n$.
Since any composition of graph homomorphisms is also a graph homomorphism, this generates the following hierarchy among color families of cliques and odd cycles.

\newpage

\[
\dots \mathcal{L}(C_{2n+1}) \subsetneq \mathcal{L}(C_{2n-1}) \subsetneq \dots \subsetneq \mathcal{L}(C_5) \subsetneq \mathcal{L}(C_3) =
\]
\[
= \mathcal{L}(K_3) \subsetneq \mathcal{L}(K_4) \subsetneq \dots \subsetneq \mathcal{L}(K_n) \subsetneq \mathcal{L}(K_{n+1}) \dots
\]

For a given graph $H$, the $H$-coloring problem is the decision problem, ``Is a given graph $H$-colorable?" In 1990, Hell and Ne\u{s}et\u{r}il showed that if $\chi(H) \leq 2$, then this problem is polynomial and if $\chi(H) \geq 3$, then it is NP-complete \cite{HNComp}.
\medskip

The concept of $H$-colorability can been extended to directed graphs.
Work has been done by Hell, Zhu and Zhou in characterizing homomorphic preimages of certain families of directed graphs, including oriented cycles \cite{Z}, \cite{HZ}, \cite{HZZ2}, oriented paths \cite{HZ2} and local acyclic tournaments \cite{HZZ}.
\medskip

In \cite{BC}, Boutin and Cockburn generalized the notion of graph homomorphisms to geometric graphs.  
A {\it geometric graph} $\overline{G}$ is a simple graph $G$ together with a straight-line drawing of  $G$  in the plane with vertices in general position (no three vertices are collinear and no three edges cross at a single point). A geometric graph $\G$ with underlying abstract graph $G$ is  called a {\it geometric realization} of $G$.
The definition below formalizes what it means for two geometric realizations of $G$ to be considered the same.
\begin{defn}
 A {\it geometric isomorphism}, denoted $f:\overline{G} \to \overline{H}$, is a  function $f:V(\overline  G) \to V( \overline  H)$ such that for all $u, v, x, y \in V(\overline  G)$,
	\begin{enumerate}
	\item  $uv \in E(\overline  G)$ if and only if  $f(u)f(v) \in E(\overline  H)$, and
	\item  $xy$ crosses $uv$ in $\overline{G}$ if and only if  $f(x)f(y)$ crosses $f(u)f(v)$ in $\overline{H}$.
	\end{enumerate} 	
\end{defn}

\noindent Relaxing the biconditionals to implications yields the following.

\begin{defn}
 A {\it geometric homomorphism}, denoted $f:\overline{G} \to \overline{H}$, is a  function $f:V(\overline  G) \to V( \overline  H)$ such that for all $u, v, x, y \in V(\overline  G)$,
	\begin{enumerate}
	\item if  $uv \in E(\G)$, then $ f(u)f(v) \in E(\overline  H)$, and
	\item  if $xy$ crosses $uv$ in $\G$, then $ f(x)f(y)$ crosses $f(u)f(v)$ in $\overline{H}$.
	\end{enumerate} 
If such a function exists, we write $\G \to \H$ and say that $\G$ is homomorphic to $\H$, or equivalently that $\G$ is a homomorphic preimage of $\H$.	
\end{defn}

 An easy consequence of this definition is that no two vertices that are adjacent or co-crossing ({\it i.e.} incident to distinct edges that cross each other) can have the same image (equivalently, can be identified) under a geometric homomorphism.
\medskip

Boutin and Cockburn define $\G$ to be $n$-{\it geocolorable}  if $\G \to \overline{K}_n$, where $\overline{K}_n$ is some geometric realization of the  $n$-clique.  The {\it geochromatic number} of $\G$, denoted $X(\G)$,  is the smallest $n$ such that $\G$ is $n$-geocolorable.
Observe that if a geometric graph of order $n$ has the property that no two of its vertices can be identified under any geometric homomorphism, then $X(\G) = n$.
The existence of multiple geometric realizations of the $n$-clique for $n > 3$ necessarily complicates the definition of geocolorability, but there is additional structure we can take advantage of.

\begin{defn}
Let $\G$ and $\widehat{G}$ be geometric realizations of  $G$.   Then set $\G \preceq \widehat{G}$ if  there exists a (vertex) injective geometric homomorphism $f:\G \to \widehat{G}$.  The set of isomorphism classes of geometric realizations of $G$ under this partial order, denoted  $\mathcal{G}$, is called the {\it homomorphism poset} of $G$. 
\end{defn}

Hence, $\G$ is $n$-geocolorable if  $\G$ is homomorphic to some element of the homomorphism poset $\mathcal{K}_n$.
In \cite{BCDM},  it is shown that $\mathcal{K}_3, \mathcal{K}_4$ and $\mathcal{K}_5$ are all chains.  Hence, for $3 \leq n \leq 5$,  $\G$ is $n$-geocolorable if and only if $\G \to \overline{K}_n$, where $\overline{K}_n$ is the last element of the chain.  By contrast, $\mathcal{K}_6$ has three maximal elements, so $\G$ is 6-geocolorable if and only if it is homomorphic to one of these three realizations.
\medskip

\begin{defn}
Let  $\mathcal{H}$ denote the homomorphism poset of geometric realizations of a simple graph $H$.  Then 
 $\G$ is $\mathcal{H}$-geocolorable if and only if $\G \to \H$ for some maximal $\H \in \mathcal{H}$.  
\end{defn}

In this paper, we provide necessary and sufficient conditions for $\G$ to be $\mathcal{C}_n$-geocolorable, where $3 \leq n \leq 5$. The structure of the homomorphism posets $\mathcal{C}_n$ for $3 \leq n \leq 5$ is given in \cite{BCDM}.  
It is worth noting that the geometric cycles are richer than than abstract cycles.  
All even cycles are homomorphically equivalent to  $K_2$, and as noted earlier, $C_{2k+1} \to C_{2\ell +1}$ if and only if $k \geq \ell$.
However, since geometric homomorphisms preserve edge crossings, and both $K_2$ and $C_3$ have only plane realizations, this is not true even for small non-plane geometric cycles, as shown in Figure~\ref{fig:notsofast}.

\begin{figure}[htbp] 
   \centering
   \includegraphics[width=3.5in]{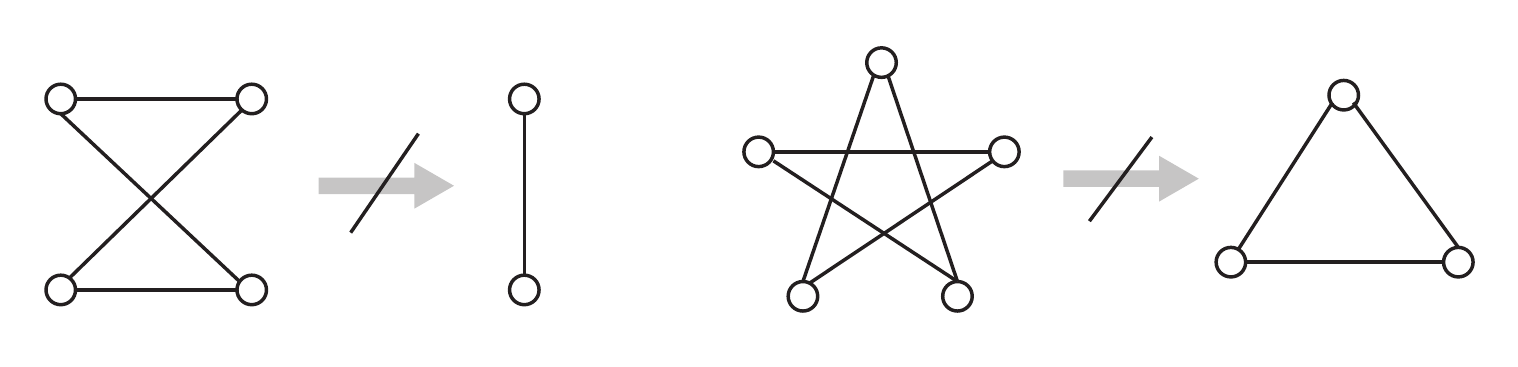} 
   \caption{$\widehat{C}_4 \not \to \overline{K_2}$ and $\widehat{C}_5 \not \to \overline{C_3}$}
   \label{fig:notsofast}
\end{figure}

\section{Edge-Crossing Graph and Thickness Edge Colorings} 

\begin{defn} \cite{BCDM}
The {\it edge-crossing graph} of a geometric graph $\G$, denoted by $EX(\G)$, is the abstract graph whose vertices correspond to the edges of $\G$, with adjacency when the corresponding edges of $\G$ cross. 
\end{defn}

Clearly, non-crossing edges of $\G$ correspond to isolated vertices of $EX(\G)$. 
In particular, $\G$ is plane if and only if $EX(\G) \to K_1$.
To focus on the crossing structure of $\G$, we let $\G_\times$ denote the geometric subgraph of $\G$ induced by its crossing edges. 
Note that $EX(\G_\times)$ is simply $EX(\G)$ with any isolated vertices removed.
From \cite{BCDM}, a geometric homomorphism $\G \to \H$ induces a geometric homomorphism $\G_\times \to \H_\times$ as well as graph homomorphisms $G \to H$ and  $EX(\G) \to EX(\H)$. 

\begin{defn} \cite{BC}
A {\it thickness edge $m$-coloring} $\e$ of a geometric graph $\G$ is a coloring of the edges of $\G$ with $m$ colors such that no two edges of the same color cross. 
The {\it thickness} of $\G$ is the minimum number of colors required for a thickness edge coloring of $\G$.
\end{defn}

From these two definitions, a thickness edge $m$-coloring $\e$ of $\G$ is a graph homomorphism $\e:EX(\G) \to K_m$.  This can be generalized as follows.

\begin{defn}
A {\it thickness edge $C_m$-coloring} $\e$ on  $\G$ is a graph homomorphism $\e: EX(\G) \to C_m$.  
\end{defn}

Observe that under a thickness edge $C_m$-coloring, edges are colored with colors numbered $1, 2, \dots, m$ such that
  colors assigned to edges that cross each other must be  consecutive mod $m$. 
Equivalently, edges of color $i$ may only be crossed by edges of colors $i-1$ and $i+1 \mod m$. 
Note also that if $\G$ has a thickness edge $C_m$-coloring for $m > 3$,  then $\G$ cannot have three mutually crossing edges.
 \medskip

\begin{defn}
Let $\e$ be a thickness edge coloring $\G$. The plane subgraph of $\G$ induced by all edges of a given color is called a {\it monochromatic} subgraph of $\G$ under $\e$.  The monochromatic subgraph corresponding to edge color $i$ is called the $i$-subgraph of $\G$ under $\e$.
\end{defn}

We assume from now on that $\G$ has no isolated vertices, which implies that every vertex belongs to at least one monochromatic subgraph of $\G$ under any thickness edge coloring.

\section{Easy Cases: $n=3$ and $n = 4$}

The smallest (simple) cycle is $C_3 = K_3$.  As noted in \cite{BC}, $\G \to \overline{K}_3$ if and only if $\G$ is a 3-colorable plane geometric graph.  
Thus $\G$ is $\mathcal{C}_3$-geocolorable if and only if 
$G$ is 3-colorable and $EX(\G)$ is 1-colorable, or more concisely,
\[
\G \to \overline{C}_3 \iff G \to K_3 \text{  and  } EX(\G) \to K_1.
\]

Next,  $C_4$ has  two geometric realizations, one plane and the other with a single crossing, 
which we denote $\overline{C}_4$ and $\widehat{C}_4$ respectively. 
Since $\overline{C}_4 \to \widehat{C}_4$, the homomorphism poset $\mathcal{C}_4$ consists of a two element chain, as shown in Figure~\ref{fig:C4chain}. Hence  $\G$ is $\mathcal{C}_4$-geocolorable if and only if  $\G \to \widehat{C}_4$. 
\medskip

\begin{figure}[htbp] 
   \centering
   \includegraphics[width=2.3in]{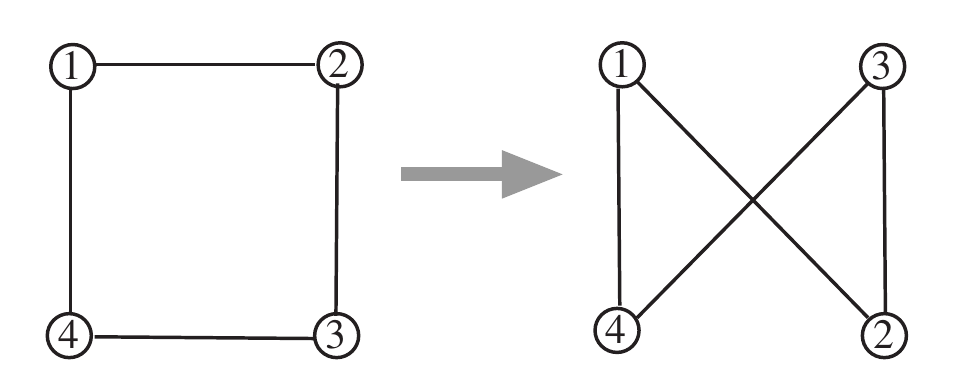} 
   \caption{$\overline{C}_4 \to \widehat{C}_4$}
   \label{fig:C4chain}
\end{figure}

If $\G \to \widehat{C}_4$, then $G \to C_4$ and $EX(\G) \to EX(\widehat{C}_4) = K_2 \cup 2K_1$. 
Since any bipartite graph is a preimage of $K_2$, 
\[
\G \to \widehat{C}_4 \implies G \to K_2 \text{  and  } EX(\G) \to K_2,
\]
which merely says that any $\mathcal{C}_4$-geocolorable geometric graph is bipartite and thickness-2.
In  \cite{BC}, Boutin and Cockburn show that the converse is false, by describing a family of bipartite, thickness-2 geometric graphs of  arbitrarily large order with the property that no two vertices can be identified under any geometric homomorphism. 
The authors do, however,  provide necessary and sufficient conditions for $\G \to \widehat{C}_4$; to describe them requires a definition.

\begin{defn}
The {\it crossing component graph} $C_\times$ of a geometric graph $\G$ is the abstract graph whose vertices correspond to the connected components $\overline{C}_1, \overline{C}_2, \dots, \overline{C}_m$ of $\G_\times$, with an edge between vertices $\overline{C}_i$ and $\overline{C}_j$ if an edge of $\overline{C}_i$ crosses an edge of $\overline{C}_j$ in $\G$.

\end{defn}

\begin{thm}\label{thm:K22} \cite{BC} A  geometric graph $\G$ is homomorphic to $\widehat{C}_4$ if and only if 
 \begin{enumerate}
 \item $\G$ is bipartite;
\item each component $\overline{C}_i$ of $\G_\times$ is a plane subgraph;
\item  $C_\times$ is bipartite.
\end{enumerate}\end{thm}

If each component of $\G_\times$ is a plane subgraph, then we can thickness edge color $\G_\times$ by coloring all the edges in a given component the same color, provided components corresponding to adjacent vertices in $C_\times$ are assigned different colors.  Moreover, in this thickness edge coloring, every vertex of $\G_\times$ appears in only one monochromatic subgraph. Conversely, if there exists a thickness edge $m$-coloring of $\G_\times$ in which the monochromatic subgraphs are vertex disjoint, then each component of $\G_\times$ must be contained in a monochromatic subgraph, and hence be plane.  Thus Theorem~\ref{thm:K22} can be rephrased more simply as follows.

\begin{thm}\label{thm:C4} A geometric graph $\G$ is $\mathcal{C}_4$-geocolorable if and only if 
\begin{enumerate}
\item $G \to K_2$, and 
\item there exists a thickness edge 2-coloring of $\Gx$ in which the two monochromatic subgraphs are vertex disjoint.
\end{enumerate}
\end{thm}

\section{ Harder Case: $n=5$}

From \cite{BCDM}, the homomorphism poset $\mathcal{C}_5$ consists of a chain of five elements, the last of which is the convex realization $\C$, as shown in Figure~\ref{fig:C5chain}.  Thus if $\G$ is $\mathcal{C}_5$-geocolorable if and only if  $\G \to \C$.
\medskip

\begin{figure}[htbp] 
   \centering
   \includegraphics[width=5in]{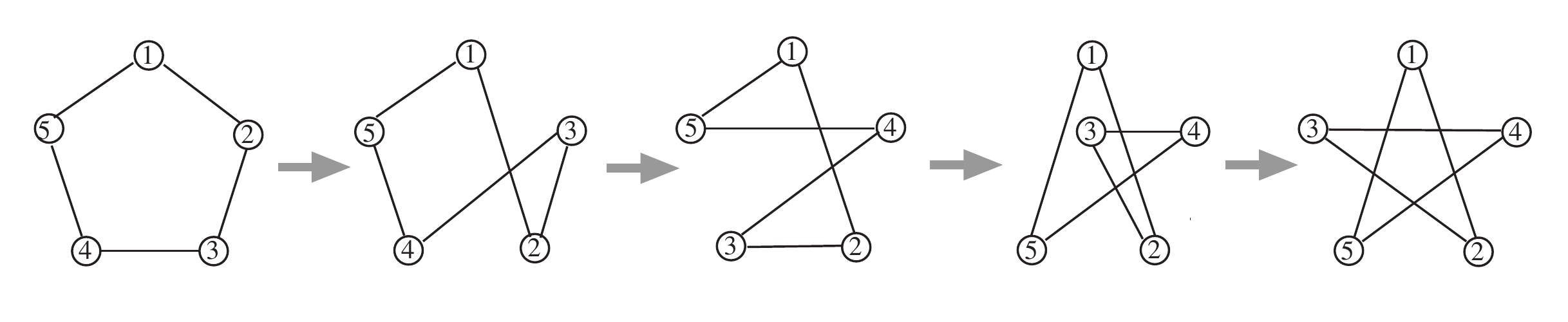} 
   \caption{Homomorphism poset $\mathcal{C}_5$}
   \label{fig:C5chain}
\end{figure}

Note that every edge of $\C$ is a crossing edge, so $(\C)_\times = \C$. Moreover, $EX(\widehat{C}_5) = C_5$; see Figure~\ref{fig:bothC5s}, where vertex labels are in bold and edge labels are in italics.
(For example, edge {\it 1} is $\{4,5\}$.) 
Note that with the labeling shown, and with the understanding that all labels are modulo 5, edge $i$ is incident with vertices $2i+2, 2i+3$ and vertex $k$ is incident with edges $3k-1, 3k+1$.  
Moreover, every vertex label is the sum of the edge labels on the vertex's two incident edges.
\medskip

\begin{figure}[h] 
   \centering
   \includegraphics[width=1.8in]{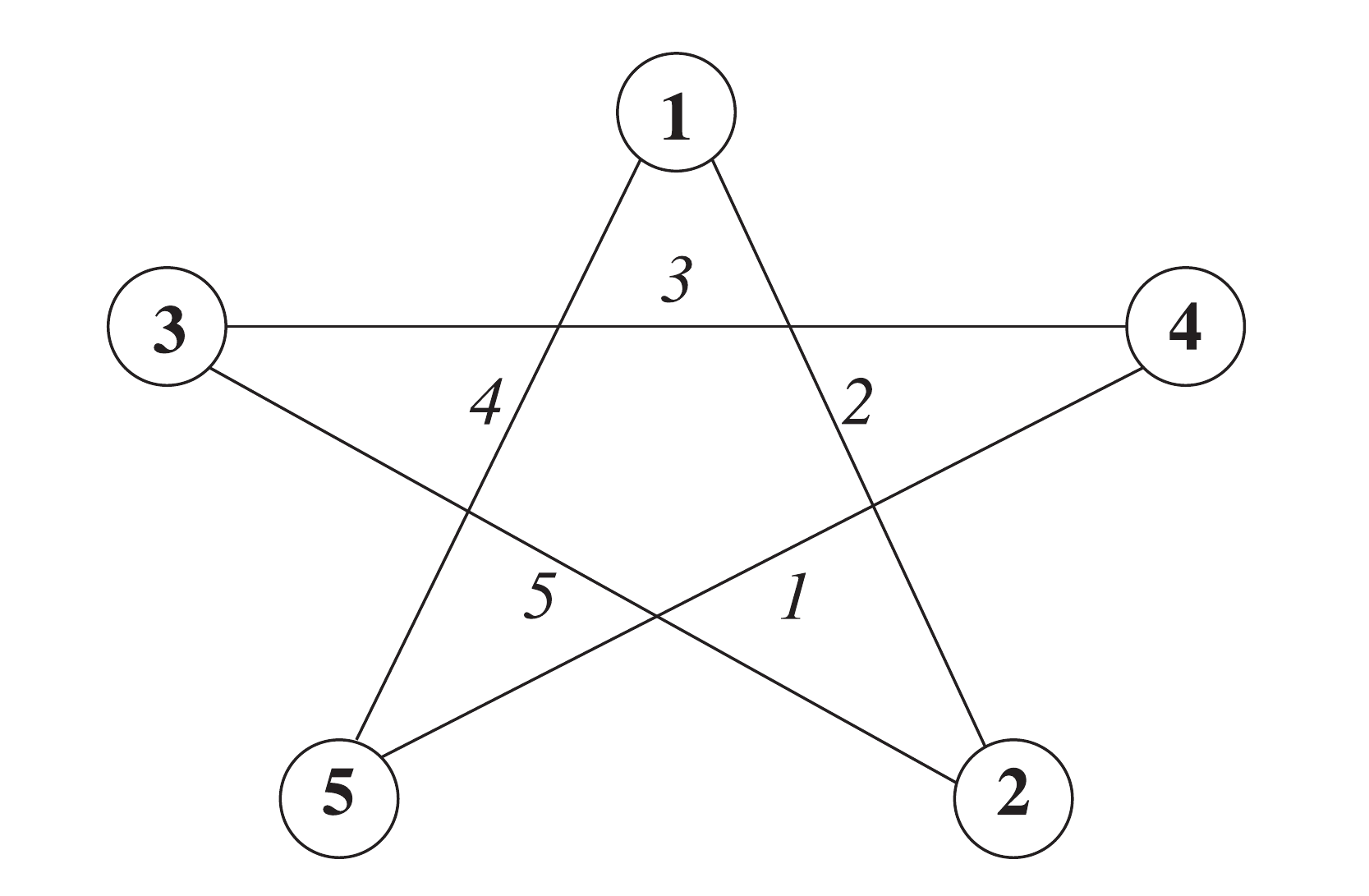} 
   \caption{$\C$ with vertex and edge labels}
   \label{fig:bothC5s}
\end{figure}

Hence if $\G$ is $\mathcal{C}_5$-geocolorable, then both $G$ and $EX(\G)$ are $C_5$-colorable.
Verifying that this necessary condition is satisfied is no easy matter, however.
Maurer {\it et al.} showed in 1981 that determining whether an abstract graph  is  $C_5$-colorable is NP-complete \cite{MSW2}.  
In 1979, Vesztergombi  related $C_5$-colorability and $5$-colorability by proving that for $G$ nonbipartite,  $G \to C_5$ if and only if  $\chi(G \boxtimes C_5) = 5$, where $\boxtimes$ denotes the strong product  \cite{V}.  Combined with Vesztergombi's result, we obtain that if $\G$ is $\mathcal{C}_5$-geocolorable and both $G$ and $EX(\G)$ are nonbipartite, then 
$ \chi(G \boxtimes C_5) = \chi (EX(\G) \boxtimes C_5) = 5$.
\medskip

However, as was the case with $n=4$,  $G \to C_5$ and $EX(\G) \to C_5$ together are not sufficient for $\G$ to be $\mathcal{C}_5$-geocolorable.  
For example,  $\G$ in Figure~\ref{fig:C5countereg} has a $C_5$-coloring (as indicated by the vertex labels, in bold) as well as a thickness edge $C_5$-coloring (as indicated by edge labels,  in italics).  
However, since any two vertices of $\G$ are either adjacent or co-crossing, no two vertices can have the same homomorphic image. 
In particular, $X(\G) = 7$. 
\medskip


\begin{figure}[htbp] 
   \centering
   \includegraphics[width=2.3in]{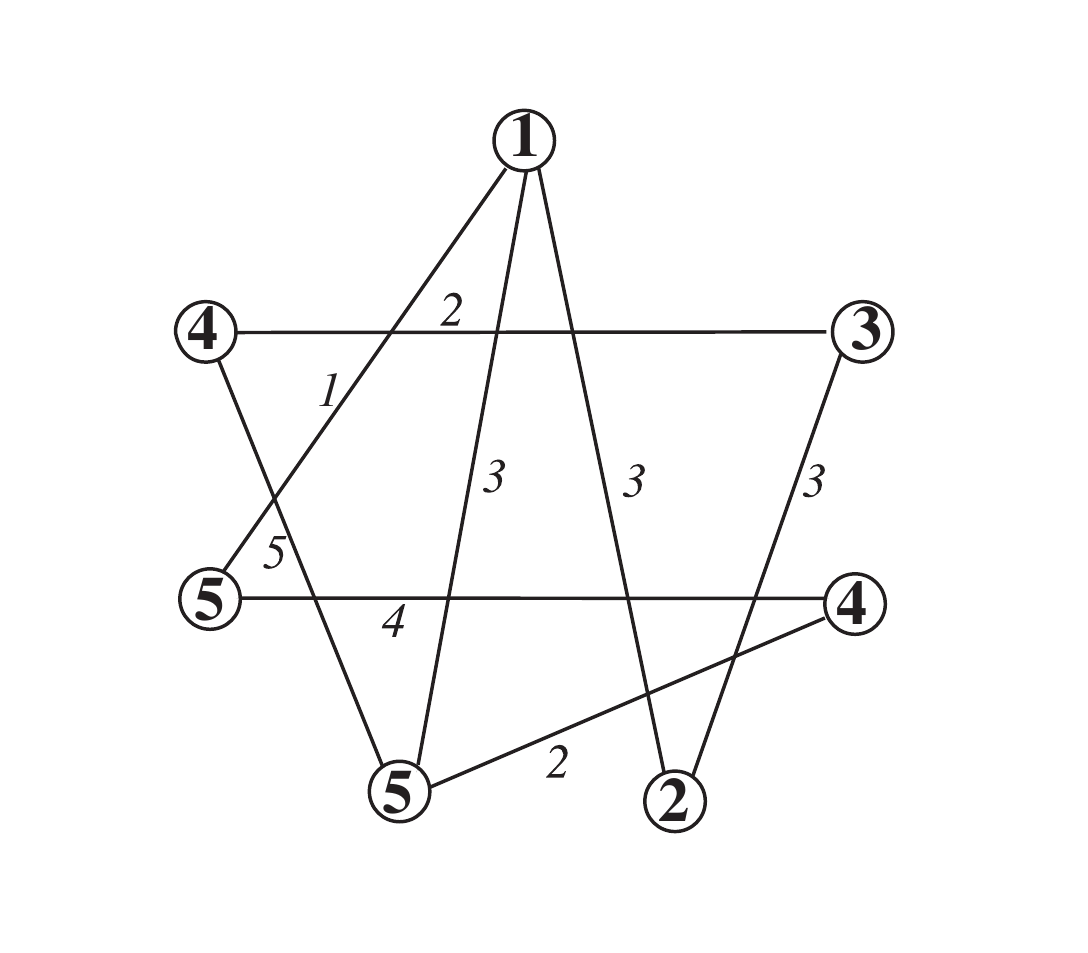} 
   \caption{$G \to C_5$ and $EX(\G) \to C_5$,  but $\G \not \to \widehat{C}_5$}
   \label{fig:C5countereg}
\end{figure}

The following theorem provides necessary and sufficient conditions for  $\G$ to be $\mathcal{C}_5$-geocolorable. Unlike Theorem~\ref{thm:C4}, the conditions involve only  thickness edge colorings, not vertex colorings.

\begin{thm}\label{thm:NSC5}
A geometric graph $\G$ is $\mathcal{C}_5$-geocolorable if and only if  there exists a thickness edge $C_5$-coloring $\e$ of $\G$ such that
\begin{enumerate}
\item any vertex of $\G$ belongs to at most two monochromatic subgraphs under $\e$; 
\item two monochromatic subgraphs can intersect ({\it i.e.} have common vertices) only if the corresponding colors are not 
consecutive mod $5$
(equivalently, the $i$-subgraph can intersect only with the $(i+2)$-subgraph and $(i+3)$-subgraph);
\item each monochromatic subgraph is bipartite and moreover, there exists a partition  in the $i$-subgraph such that all vertices also in the $(i+2)$-subgraph (if any) belong to one partite set, and all those also in the $(i+3)$-subgraph (if any) belong to the other.
\end{enumerate}
\end{thm}

\begin{proof}
Assume $f:\G \to \C$.  This induces an abstract graph homomorphism $EX(\G) \to C_5$.  We can pull back the 
edge colors shown in Figure~\ref{fig:bothC5s} to obtain 
a thickness edge $C_5$-coloring on $\G$.
Note that  every vertex of $\C$ is incident to edges of exactly two colors that are not consecutive mod $5$, so under $\e$, $\G$ must satisfy conditions (1) and (2).
\medskip

Since the  $i$-subgraph of $\G$ maps onto the single $i$-colored edge $\{2i+2, 2i+3\}$ of $\C$, by  transitivity  it is homomorphic to $K_2$ and is thus bipartite.  
Moreover, all vertices also in the $(i+2)$-subgraph get mapped to $2i+2$ and all vertices also in the $(i+3)$-subgraph get mapped to $2i+3$.  Hence, $\G$ satisfies (3).
\medskip

For the converse, assume $\G$ has a thickness edge $C_5$-coloring $\e$ satisfying conditions (1) - (3). 
First label all vertices that are in two monochromatic subgraphs with the sum of the two corresponding colors mod $5$.
To label the vertices that are in only one monochromatic subgraph, say the $i$-subgraph, first break this bipartite subgraph into connected components.  
By condition (3), if a component has vertices that have already been labeled, then we can label the remaining vertices either $2i+2$ or $2i+3$ according to the partite set they are in.  
If a component of the $i$-subgraph has no vertices that are already labeled, then we can arbitrarily assign the the label $2i+2$  to vertices in one partite set and $2i+3$ to those in the other.
\medskip

To show that $f$ is a graph homomorphism, let $u, v \in V(\G)$ be adjacent vertices.  
WLOG edge $uv$ is colored $i$, so $u$ and $v$ both belong to the $i$-subgraph.  
WLOG again, $f(u) = 2i+2$ and $f(v) = 2i+3 \mod 5$.  Since these are consecutive  mod $5$,  $f(u)$ and $f(v)$  are adjacent in $\C$.
\medskip

Next we show that $f$ is a geometric homomorphism.  Suppose that in $\G$, edge $ux$ crosses edge $vy$.  
Since $\e$ is a thickness edge $C_5$-coloring, crossing edges must be assigned consecutive colors mod $5$.
Assume $ux$ is colored $i$ and  $vy$ is colored $i+1$. 
Then WLOG, $f(u) = 2i+2$, $f(x) = 2i+3$, $f(v) = 2(i+1)+2 = 2i+4$ and $f(y) = 2(i+1)+3 = 2i$.  Set $j=2i+2$ and notice that all pairs of edges of the form $\{j, j+1\}$ and $\{j+2, j+3\}$ cross in $\C$.
\end{proof}

We show how this theorem can be applied to $\G$ in Figure~\ref{fig:C5countereg}.
The thickness edge $C_5$-coloring shown violates condition (1) of the theorem because both vertices of degree $3$ are incident to edges of 3 different colors.
In fact, no thickness edge $C_5$-coloring on this geometric graph will satisfy all $3$ conditions of Theorem~\ref{thm:NSC5}.  
We being by noting that in any thickness edge $C_5$-coloring, any $5$-cycle of crossings will have to involve all $5$ colors.  WLOG, we can start with the edge colors shown in Figure~\ref{fig:C5counter1}.
\medskip

\begin{figure}[h] 
   \centering
   \includegraphics[width=2.3in]{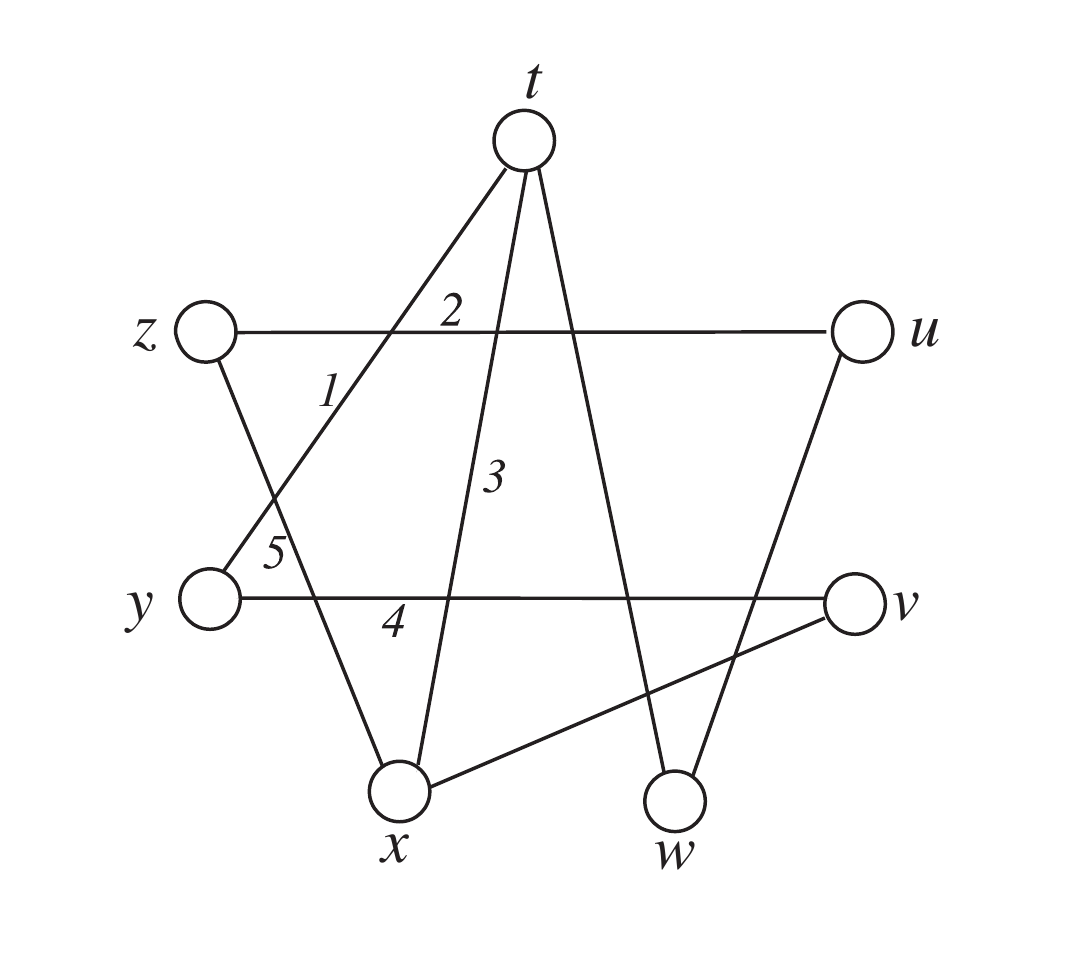} 
   \caption{Recoloring $\G$}
   \label{fig:C5counter1}
\end{figure}

Since edge $tw$ crosses edges of colors $2$ and  $4$, it must be colored $3$.
Next, vertex $u$ is incident to an edge colored $2$, so to satisfy condition (2), edge $uw$ must be colored either $2$, $4$ or $5$.
Since $uw$ crosses  $vy$ which is colored $4$,  $uw$ must be colored $5$.
Edge $xv$ crosses edges of color $3$ and $5$, so it must be colored $4$.
However, now vertex $x$ appears in $3$ monochromatic subgraphs, violating condition (1).
\medskip

Consider the graph $\H$ obtained from $\G$  by deleting $xv$, shown in Figure~\ref{fig:C5counter2}, with edges colored as required in the previous paragraph.
We still have a problem; $u$ and $z$ are vertices in the $2$-subgraph belonging also to the $5$-subgraph, yet they are an odd distance apart. Hence $\H$ is also not $\mathcal{C}_5$-colorable.

\begin{figure}[h] 
   \centering
   \includegraphics[width=2.3in]{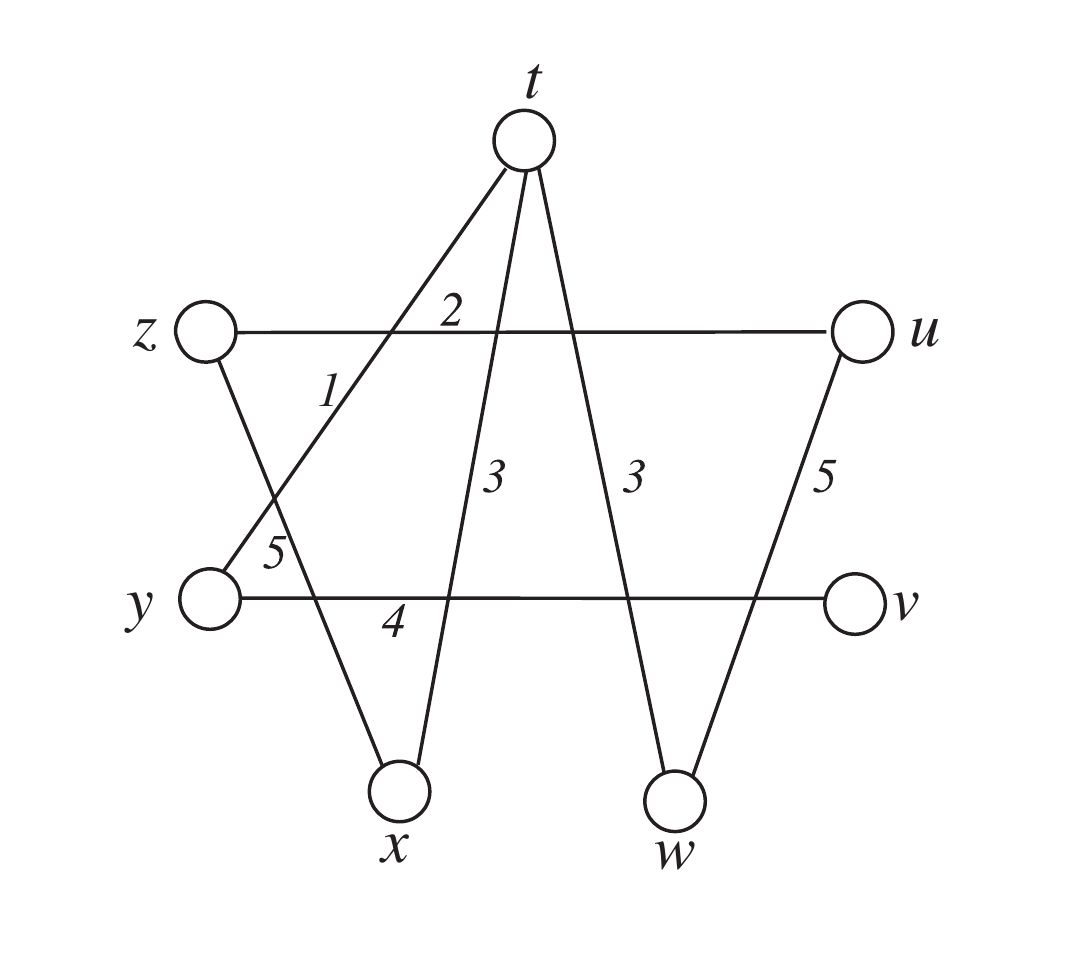} 
   \caption{$\H$}
   \label{fig:C5counter2}
\end{figure}

\section{5-geocolorability}

Recall that $\G$ is n-geocolorable if and only if $\G$ is homomorphic to some realization of $K_n$.  In \cite{BC}, Boutin and Cockburn give a set of necessary but not sufficient conditions (Theorem 4), as well as a set of sufficient but not necessary conditions (Corollary 5.1) for $\G$ to be 4-geocolorable. Finding necessary and sufficient conditions for a geometric graph to be 5-geocolorable is likely to be even more difficult. However, the work in the previous section allows us to make some progress.
\medskip

From \cite{BCDM}, the homomorphism poset $\mathcal{K}_5$ is chain of length 3, with last element $\widehat{K_5}$,  shown in Figure~\ref{fig:K5poset}. Hence $\G$ is 5-geocolorable  if and only if $\G \to \widehat{K_5}$.
\medskip

\begin{figure}[htbp] 
   \centering
   \includegraphics[width=4in]{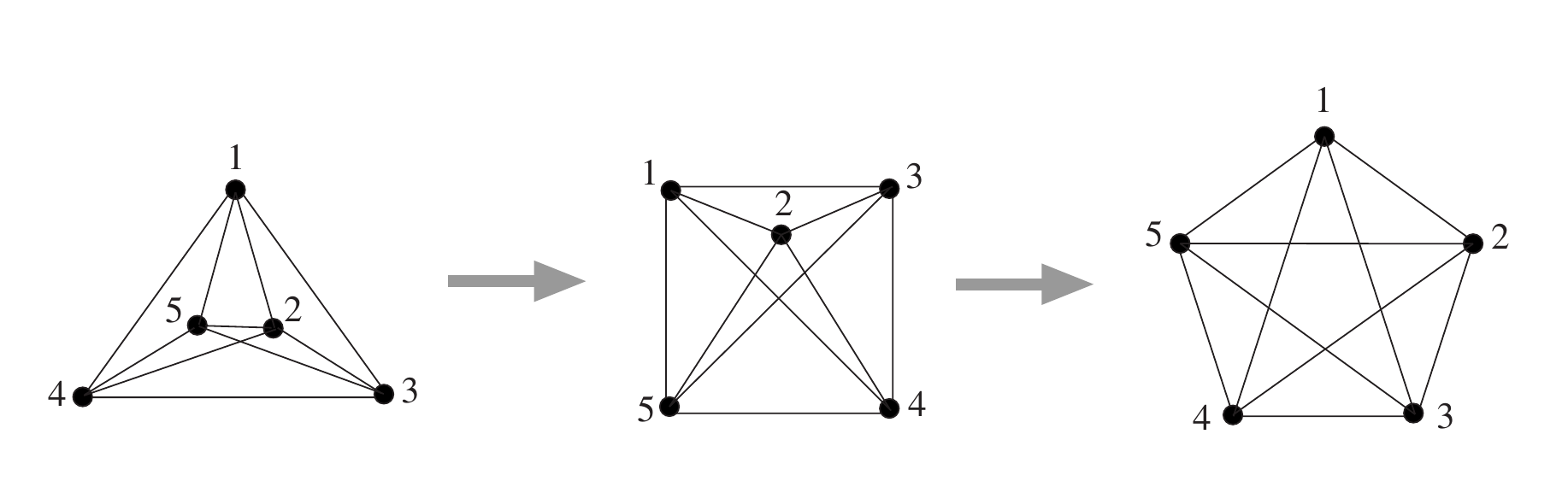} 
   \caption{The homomorphism poset $\mathcal{K}_5$.}
   \label{fig:K5poset}
\end{figure}

From \cite{BC}, $\G \to \H$ implies $\G_\times \to \H_\times$, although the converse is false.  Hence if $\G$ is 5-geocolorable, then $\G_\times \to \C$; equivalently, if $\G$ is 5-geocolorable then $\G_\times$ is $\mathcal{C}_5$-geocolorable.  The contrapositive is, of course, if $\G_\times$ is not $\mathcal{C}_5$-geocolorable, then $\G$ is not 5-geocolorable.

\section{Future Work}

Finding necessary and sufficient conditions for a geometric graph to be $\mathcal{C}_6$-colorable is complicated by the fact that the homomorphism poset $C_6$ has two maximal elements, shown in Figure~\ref{fig:2C6s} (see \cite{BCDM}). The  one on the left  is bipartite and thickness-2, while the one on the right is bipartite and thickness-3. We investigate these in a future paper.

\begin{figure}[htbp] 
   \centering
   \includegraphics[width=3in]{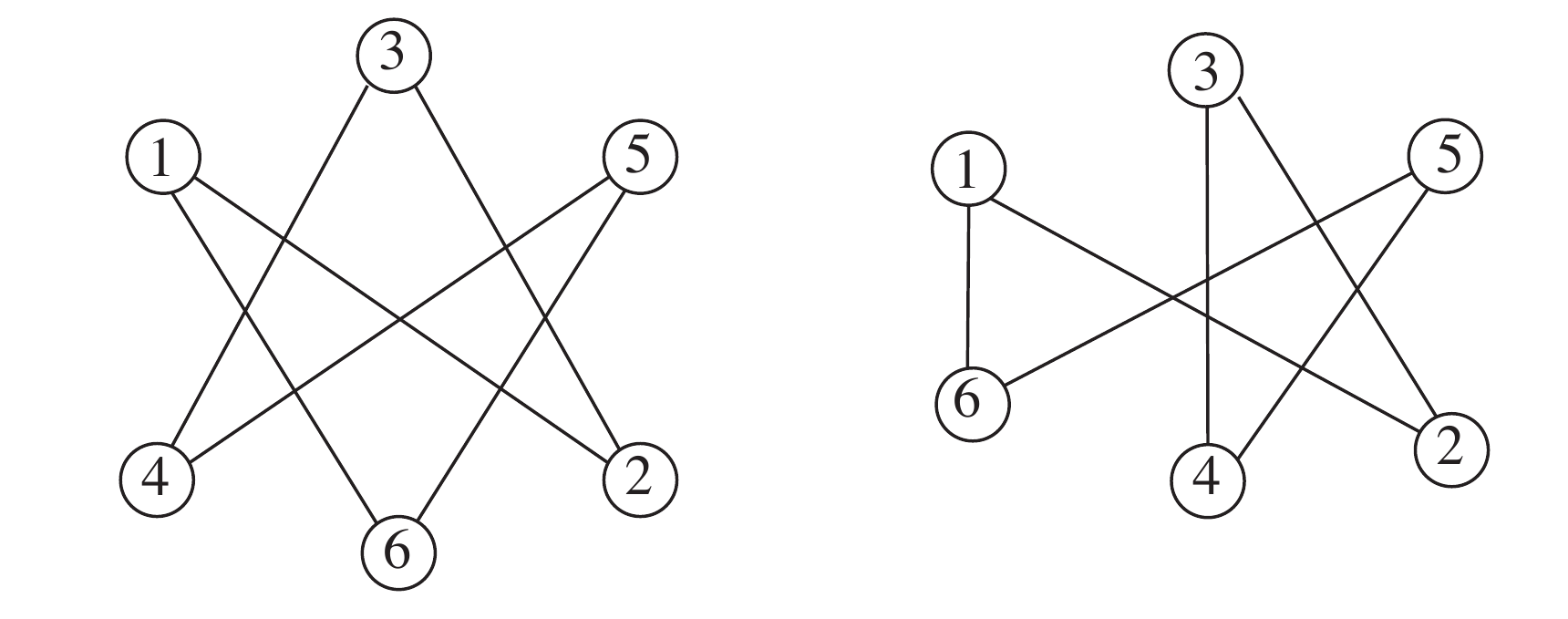} 
   \caption{Two maximal elements of $\mathcal{C}_6$.}
   \label{fig:2C6s}
\end{figure}

\bibliographystyle{plain}

\bibliography{Bipartite}

\end{document}